\documentclass[twoside,english,final, 1p, times]{elsarticle}
\usepackage[T1]{fontenc}
\usepackage[latin9]{inputenc}
\usepackage{listings}
\usepackage{geometry}
\usepackage{fancyhdr}
\usepackage{color}
\usepackage{babel}
\usepackage{float}
\usepackage{amsthm}
\usepackage{amsmath}
\usepackage{amssymb}
\usepackage{graphicx}
\usepackage{esint}
\usepackage[unicode=true,
 bookmarks=true,bookmarksnumbered=false,bookmarksopen=false,
 breaklinks=true,pdfborder={0 0 1},backref=false,colorlinks=true]
 {hyperref}
\hypersetup{pdftitle={Quantization},
 pdfauthor={Alois Pichler}}

\makeatletter

\floatstyle{ruled}
\newfloat{algorithm}{tbp}{loa}
\providecommand{\algorithmname}{Algorithm}
\floatname{algorithm}{\protect\algorithmname}

\numberwithin{equation}{section}
\numberwithin{figure}{section}
\theoremstyle{plain}
\newtheorem{thm}{\protect\theoremname}
  \theoremstyle{remark}
  \newtheorem{rem}[thm]{\protect\remarkname}
  \theoremstyle{remark}
  \newtheorem*{rem*}{\protect\remarkname}

\journal{nowhere}
\renewcommand{\Re}{\operatorname{Re}}
\renewcommand{\Im}{\operatorname{Im}}

\AtBeginDocument{
  
}

\makeatother

  \providecommand{\remarkname}{Remark}
\providecommand{\theoremname}{Theorem}

\begin{document}

\title{On A Rapidly Converging Series For The Riemann Zeta Function}

\author{Alois Pichler}

\address{Department of Statistics and Operations Research, University of Vienna,
Austria, Universitätsstraße 5, 1010 Vienna}

\ead{alois.pichler@univie.ac.at}
\begin{abstract}
To evaluate Riemann's zeta function is important for many investigations
related to the area of number theory, and to have quickly converging
series at hand in particular. We investigate a class of summation
formulae and find, as a special case, a new proof of a rapidly converging
series for the Riemann zeta function. The series converges in the
entire complex plane, its rate of convergence being significantly
faster than comparable representations, and so is a useful basis for
evaluation algorithms. The evaluation of corresponding coefficients
is not problematic, and precise convergence rates are elaborated in
detail. The globally converging series obtained allow to reduce Riemann's
hypothesis to similar properties on polynomials. And interestingly,
Laguerre's polynomials form a kind of leitmotif through all sections.\end{abstract}
\maketitle

\section{Introduction and Definitions}

\subsection{Confluent Hypergeometric Functions}

Confluent hypergeometric functions are special hypergeometric functions,
sometimes called also Kummer's function of first and second kind.
They are linear independent solutions of Kummer's differential equation
\[
z\, y''\left(z\right)+\left(b-z\right)\, y'\left(z\right)-a\, y\left(z\right)=0.
\]
The first solution, Kummer's function of the first kind, is usually
given as a globally converging power series 
\begin{equation}
M\left(a;b;z\right):=\sum_{i=0}^{\infty}\frac{{a+i-1 \choose i}}{{b+i-1 \choose i}}\frac{z^{i}}{i!}\label{Kummer1}
\end{equation}
(Kummer's function of the first kind), and Kummer's function of the
second kind is often given by 
\begin{equation}
U\left(a;b;z\right):=\frac{\Gamma\left(1-b\right)}{\Gamma\left(1-b+a\right)}M\left(a;b;z\right)-\frac{\Gamma\left(b-1\right)}{\Gamma\left(a\right)}z^{1-b}M\left(a-b+1;2-b;z\right)\label{Kummer2}
\end{equation}

Kummer's transformation states that $M\left(a;b;z\right)=e^{z}\, M\left(b-a;b;-z\right)$,
which subsequently leads to the identity $U\left(a;b;z\right)=z^{1-b}\, U\left(1+a-b;2-b;z\right)$.

Another solution of this differential equation -- which turns out
to be identical to $U$ and thus is simply another representation
of $U$ -- is obtained as an integral representation by the Laplace
transform 
\begin{align}
U\left(a;b;z\right) & =\frac{1}{\Gamma\left(a\right)}\int_{0}^{\infty}e^{-zt}t^{a-1}\left(1+t\right)^{b-a-1}\mathrm{d}t\nonumber \\
 & =\frac{z^{1-b}}{\Gamma\left(1+a-b\right)}\int_{0}^{\infty}e^{-zt}\frac{t^{a-b}}{\left(1+t\right)^{a}}\mathrm{d}t.\label{U als Integral}
\end{align}

Some particular and frequently used confluent hypergeometric functions
are the upper and lower incomplete Gamma function, which have the
representations 
\begin{align}
\Gamma\left(s,z\right) & :=\int_{z}^{\infty}t^{s-1}e^{-t}\mathrm{d}t\nonumber \\
 & =e^{-z}U\left(1-s,1-s,z\right)=e^{-z}z^{s}U\left(1,1+s,z\right)\label{eq:GammaUpper}
\end{align}
and 
\begin{align}
\gamma\left(s,z\right) & :=\int_{0}^{z}t^{s-1}e^{-t}\mathrm{d}t=\sum_{k=0}\frac{\left(-1\right)^{k}}{k!}\frac{z^{s+k}}{s+k}\nonumber \\
 & =\frac{z^{s}}{s}M\left(s,s+1,-z\right)=\frac{z^{s}}{s}e^{-z}M\left(1,s+1,z\right).\label{eq:GammaLower}
\end{align}

\subsection{Laguerre's Polynomials}

For $a$ a negative integer the defining series for $M$ reduces to
a polynomial, which turns out to be closely related to Laguerre's
polynomials: The explicit representation is
\begin{align}
L_{i}^{(\alpha)}\left(z\right) & ={i+\alpha \choose i}\, M\left(-i,\alpha+1,z\right)=\sum_{j=0}^{i}\left(-1\right)^{j}{i+\alpha \choose i-j}\frac{z^{j}}{j!}.\label{Laguerre}
\end{align}

In view of \eqref{Kummer2} Laguerre's polynomial may be given by
Kummer's second function as well, that is $L_{i}^{(\alpha)}\left(z\right)=\frac{\left(-1\right)^{i}}{i!}U\left(-i,\alpha+1,z\right)$.
This somehow suggests that Laguerre's polynomial are somewhat in between
of both solutions $M$ and $U$ of Kummer's differential equation.
This is central in our investigations and reflected in the results
of the next sections.

In addition to that it is well-know that Laguerre's polynomials are
orthogonal with respect to the weight-function $z^{\alpha}e^{-z}$;
more precisely we find that 
\begin{equation}
\int_{0}^{\infty}z^{\alpha}e^{-z}L_{i}^{(\alpha)}\left(z\right)L_{j}^{(\alpha)}\left(z\right)\mathrm{d}z=\begin{cases}
0 & i\ne j\\
{i+\alpha \choose i}\Gamma\left(\alpha+1\right) & i=j,
\end{cases}\label{inneres Produkt}
\end{equation}
Laguerre's polynomials thus are orthogonal with respect to the inner
product 
\[
\left\langle g|\, f\right\rangle :=\int_{0}^{\infty}\frac{z^{\alpha}e^{-z}}{\Gamma\left(\alpha+1\right)}g\left(z\right)f\left(z\right)\mathrm{d}z.
\]
As a result of the classical theory on Hilbert spaces we may expand
a function $f$ in a series with respect to this orthogonal basis,
provided that $\left\Vert f\right\Vert _{2}:=\sqrt{\left\langle f|\, f\right\rangle }<\infty$.
The function has the expansion $f\left(z\right)=\sum_{i=0}f_{i}^{(\alpha)}L_{i}^{(\alpha)}\left(z\right)$
(and is therefore in the closed hull of all $L_{i}^{\left(\alpha\right)}$),
its coefficients taking the explicit form 
\begin{equation}
f_{i}^{\alpha}=\int_{0}^{\infty}\frac{L_{i}^{(\alpha)}\left(z\right)}{{i+\alpha \choose i}}\frac{z^{\alpha}e^{-z}}{\Gamma\left(\alpha+1\right)}f\left(z\right)\mathrm{d}z\label{Koeffizient}
\end{equation}
(cf. \eqref{inneres Produkt}). Conversely, the norm can be recovered
from the function's coefficients, as $\left\Vert f\right\Vert _{2}^{2}=\sum_{i=0}{i+\alpha \choose i}\left|f_{i}^{\alpha}\right|^{2}$.

An elementary example of such a representation in explicit terms is
\begin{equation}
e^{-tz}=\frac{1}{\left(1+t\right)^{\alpha+1}}\sum_{i=0}L_{i}^{(\alpha)}\left(z\right)\left(\frac{t}{1+t}\right)^{i},\label{Exponential als LaguerreReihe}
\end{equation}
as can be verified straight forward by evaluating the respective integrals
for the coefficients \eqref{Koeffizient}. Notably, this series converges
point-wise if $t>-\frac{1}{2}$ ($\left\Vert f\right\Vert _{2}<\infty$),
even if $\alpha\le-1$.

To give a reference of these classical ingredients aggregated in this
section above we would like to refer to the standard work \citep{abramowitz+stegun}.

\section{Fourier Series}

Kummer's functions allow some explicit representation as series of
Laguerre polynomials.
\begin{thm}[Expansion of Kummer's functions in terms of Laguerre polynomials]
\label{Kummer als Laguerre}

Suppose that $\Re\left(b-a\right)>0$ and $\Re\left(\alpha-2b\right)>-\frac{5}{2}$,
then
\[
M\left(a;b;z\right)=\frac{{b-1 \choose a}}{{b-\beta-1 \choose a}}\sum_{i=0}\left(-1\right)^{i}L_{i}^{(\beta-i)}\left(z\right)\frac{{-a \choose i}}{{\beta-b \choose i}}
\]
and 
\begin{equation}
U\left(a;b;z\right)=\frac{\left(1+\alpha-b\right)!}{\left(1+\alpha-b+a\right)!}\sum_{i=0}L_{i}^{(\alpha)}\left(z\right)\frac{{-a \choose i}}{{b-a-\alpha-2 \choose i}};\label{eq:UalsLaguerre}
\end{equation}
moreover, 
\begin{equation}
\Gamma\left(s;z\right)=z^{s}e^{-z}\sum_{k=0}\frac{L_{k}^{\left(\alpha\right)}\left(z\right)}{\left(k+1\right){k+1+\alpha-s \choose k+1}}\label{eq:Gamma als Laguerre-Reihe}
\end{equation}
for $\Re\left(s-\frac{\alpha}{2}\right)<\frac{1}{4}$.\end{thm}
\begin{proof}
As for the proof notice first that $\frac{z^{i}}{i!}=\left(-1\right)^{i}\sum_{j=0}^{i}{-\beta \choose i-j}L_{j}^{\left(\beta-j\right)}\left(z\right)$,
which is a kind of converse of \eqref{Laguerre} and verified straight
forward by comparing the coefficients of respective powers of $z$.
Substituting this into \eqref{Kummer1}, interchanging the order of
summation and employing Gauss' hypergeometric theorem (cf. \citep{Koepf})
gives 
\begin{align*}
M\left(a;b;z\right) & =\sum_{i=0}^{\infty}\frac{{a+i-1 \choose i}}{{b+i-1 \choose i}}\left(-1\right)^{i}\sum_{j=0}^{i}{-\beta \choose i-j}L_{j}^{(\beta-j)}\left(z\right)\\
 & =\sum_{j=0}^{\infty}L_{j}^{(\beta-j)}\left(z\right)\sum_{i=j}^{\infty}\left(-1\right)^{i}\frac{{a+i-1 \choose i}}{{b+i-1 \choose i}}{-\beta \choose i-j}\\
 & =\sum_{j=0}L_{j}^{(\beta-j)}\left(z\right)\left(-1\right)^{j}\frac{\Gamma\left(b\right)\Gamma\left(b-a-\beta\right)\Gamma\left(a+j\right)}{\Gamma\left(a\right)\Gamma\left(b-a\right)\Gamma\left(b-\beta+j\right)}\\
 & =\sum_{j=0}L_{j}^{(\beta-j)}\left(z\right)\left(-1\right)^{j}\frac{{b-1 \choose a}}{{b-\beta-1 \choose a}}\frac{{a+j-1 \choose j}}{{b-\beta-1+j \choose j}}\\
 & =\frac{{b-1 \choose a}}{{b-\beta-1 \choose a}}\sum_{j=0}\left(-1\right)^{j}L_{j}^{(\beta-j)}\left(z\right)\frac{{-a \choose j}}{{\beta-b \choose j}},
\end{align*}
which is the desired assertion. As regards convergence it follows
from the elementary recursion 
\begin{equation}
L_{j}^{\left(\beta-j\right)}\left(z\right)=-\left(1+\frac{z-\beta-1}{j}\right)L_{j-1}^{\left(\beta-j+1\right)}\left(z\right)-\frac{z}{j}L_{j-2}^{\left(\beta-j+2\right)}\left(z\right)\label{eq:RecursionLaguerre}
\end{equation}
that $\left(-1\right)^{j}L_{j}^{\left(\beta-j\right)}\left(z\right)=\mathcal{O}\left(j^{-\beta-1}\right)$,
a more thorough analysis even discloses that $\frac{L_{j}^{\left(\beta-j\right)}\left(z\right)}{{\beta \choose j}}=e^{z}\left(1+\mathcal{O}\left(\frac{1}{j}\right)\right)$.
Hence, $\left(-1\right)^{j}L_{j}^{(\beta-j)}\left(z\right)\frac{{-a \choose j}}{{\beta-b \choose j}}=\mathcal{O}\left(j^{-\beta-1}\frac{j^{a-1}}{j^{b-\beta-1}}\right)=\mathcal{O}\left(\frac{1}{j^{b-a+1}}\right)$,
and the series converges -- irrespective of $\beta$ -- if $\Re\left(b-a\right)>0$. 

To verify the second statement notice that 
\begin{align*}
U_{i}^{\alpha} & =\int_{0}^{\infty}\frac{L_{i}^{(\alpha)}\left(z\right)}{{i+\alpha \choose i}}\frac{z^{\alpha}e^{-z}}{\Gamma\left(\alpha+1\right)}U\left(a;b;z\right)\mathrm{d}z\\
 & =\int_{0}^{\infty}\frac{L_{i}^{(\alpha)}\left(z\right)}{{i+\alpha \choose i}}\frac{z^{\alpha}e^{-z}}{\Gamma\left(\alpha+1\right)}\frac{1}{\Gamma\left(a\right)}\int_{0}^{\infty}e^{-zt}t^{a-1}\left(1+t\right)^{b-a-1}\mathrm{d}t\mathrm{d}z\\
 & =\int_{0}^{\infty}\frac{t^{a-1}}{\Gamma\left(a\right)}\left(1+t\right)^{b-a-1}\int_{0}^{\infty}\frac{L_{i}^{(\alpha)}\left(z\right)}{{i+\alpha \choose i}}\frac{z^{\alpha}e^{-z}}{\Gamma\left(\alpha+1\right)}e^{-zt}\mathrm{d}z\mathrm{d}t\\
 & =\frac{1}{\Gamma\left(a\right)}\int_{0}^{\infty}t^{a-1}\left(1+t\right)^{b-a-1}\frac{t^{i}}{\left(1+t\right)^{\alpha+i+1}}\mathrm{d}t,
\end{align*}
which is a consequence of \eqref{U als Integral} and \eqref{Exponential als LaguerreReihe}.
The latter integral is a beta function, we thus continue and find
the coefficient

\begin{align*}
U_{i}^{\alpha} & =\frac{1}{\Gamma\left(a\right)}\int_{0}^{\infty}t^{a-1+i}\left(1+t\right)^{b-a-\alpha-i-2}\mathrm{d}t\\
 & =\frac{1}{\Gamma\left(a\right)}\frac{\Gamma\left(a+i\right)\Gamma\left(\alpha+2-b\right)}{\Gamma\left(i+a+\alpha+2-b\right)}\\
 & =\frac{\left(\alpha+1-b\right)!}{\left(i+a+\alpha+1-b\right)!}\frac{{a+i-1 \choose i}}{{i+a+\alpha+2-b \choose i}}
\end{align*}
which is the respective coefficient for $U$.

Convergence is more difficult compared to Kummer's function of the
first kind. However, we will show below (Theorem \ref{Laguerre Asymptotics})
that $\limsup_{i}\frac{L_{i}^{(\alpha)}\left(z\right)}{i^{\frac{\alpha}{2}-\frac{1}{4}}}<\infty$
and the series thus convergences if $\Re\left(\alpha-2b\right)>-\frac{5}{2}$.\end{proof}
\begin{rem}
Identity \eqref{eq:UalsLaguerre} can be found for the special case
$\alpha=b-1$ in \citep{Erdelyi1953}.
\end{rem}
To evaluate the series above using Laguerre polynomials it is necessary
to have good evaluations of Laguerre polynomials at hand for given
parameters $\alpha$ and $z$. Moreover, numerical approximations
should be sufficiently good and the computation stable. Although there
are explicit expressions or Horner's scheme available to evaluate
Laguerre polynomials, the resulting algorithms usually behave unstable
very soon. 

We have found the algorithms described much better, adaptations even
allow to evaluate and then successively store intermediary results.
However, they should not be interchanged, as this will cause numerical
instability again.

\begin{algorithm}
\caption{to evaluate $L_{i}^{\left(\alpha\right)}\left(z\right)$, based on
$L_{i}^{\left(\alpha\right)}\left(z\right)=\left(2+\frac{\alpha-1-z}{i}\right)L_{i-1}^{\left(\alpha\right)}\left(z\right)-\left(1+\frac{\alpha-1}{i}\right)L_{i-2}^{\left(\alpha\right)}\left(z\right)$.}
\begin{lstlisting}
L1:=0; Laguerre:= 1
For j:= 1 to i
	L0:= L1; L1:= Laguerre;
	Laguerre:= ((2* j+ alpha- 1- z)* L1- (j+ alpha- 1)* L0)/ j;
Next j
Return Laguerre
\end{lstlisting}
\end{algorithm}

\begin{algorithm}
\caption{to evaluate $L_{i}^{\left(\beta-i\right)}\left(z\right)$, based on
\eqref{eq:RecursionLaguerre}}
\begin{lstlisting}
L1:=0; Laguerre:= 1
For j:= 1 to i
	L0:= L1; L1:= Laguerre;
	Laguerre:= ((beta+ 1- j- z)* L1- z* L0)/ j;
Next j
Return Laguerre
\end{lstlisting}
\end{algorithm}

\section{Continuous Fourier Transform and Poisson Summation Formula}

\subsection{Continuous Fourier Transform.}

It is well-known that Poisson's summation formula provides an efficient
tool to evaluate sums, some authors dedicate entire chapters to these
summation techniques, see for instance \citep{IwaniecKowalski}. To
apply these effective summation identities we need to have a good
expression for the functions involved at hand, which involve the continuous
Fourier transform. 

In literature there occur a few variants for the continuous Fourier
transform of a function $f$, which are interchanged frequently; for
our purposes it is most convenient to state 
\[
\hat{f}\left(k\right):=\mathcal{F}\left(f\right)\left(k\right):=\int_{-\infty}^{\infty}f\left(z\right)e^{-2\pi ikz}\mathrm{d}z
\]
as a definition. 

We will sometimes abuse the notation just introduced and improperly
write $\mathcal{F}\, f\left(z\right)$ shortly for $\mathcal{F}\left(f\left(z\right)\right)\left(z\right)$,
that is to say we use the argument $z$ for both functions -- $f$
and $\mathcal{F}\left(f\right)$ -- synonymously.

The following result establishes that Laguerre's polynomials, as well
as Kummer's functions are each others Fourier transform in the following
sense:
\begin{thm}[Fourier transform of Kummer's functions]
$\,$
\begin{enumerate}
\item $\mathcal{F}\, e^{-z^{2}\pi}L_{i}^{\left(\alpha\right)}\left(z^{2}\pi\right)=\left(-1\right)^{i}e^{-z^{2}\pi}L_{i}^{\left(-i-\alpha-\frac{1}{2}\right)}\left(z^{2}\pi\right)$,
\item $\mathcal{F}\, e^{-z^{2}\pi}U\left(a,b,z^{2}\pi\right)=\frac{\Gamma\left(\frac{3}{2}-b\right)}{\Gamma\left(a+\frac{3}{2}-b\right)}e^{-z^{2}\pi}M\left(a,a+\frac{3}{2}-b,z^{2}\pi\right)$,
\item $\mathcal{F}\, e^{-z^{2}\pi}M\left(a,b,z^{2}\pi\right)=\frac{\Gamma\left(b\right)}{\Gamma\left(b-a\right)}e^{-z^{2}\pi}U\left(a,a+\frac{3}{2}-b,z^{2}\pi\right)$.
\end{enumerate}
\end{thm}
\begin{proof}
Notice first, that $e^{-z^{2}\pi}$ is an eigenfunction of the Fourier
transform $\mathcal{F}$, which already covers the desired statement
for $i=0$. Now recall that -- due to integration by parts -- $\mathcal{F}\left(\frac{\mathrm{d}^{j}}{\mathrm{d}z^{j}}f\left(z\right)\right)=\left(2\pi iz\right)^{j}\mathcal{F}\left(f\left(z\right)\right)$
and observe that differentiating this first eigenfunction involves
Laguerre's polynomials \emph{again}, as $\frac{\mathrm{d}^{2j}}{\mathrm{d}z^{2j}}e^{-z^{2}\pi}=\left(-4\pi\right)^{j}j!e^{-z^{2}\pi}L_{j}^{\left(-\frac{1}{2}\right)}\left(z^{2}\pi\right)$.
Hence, 
\begin{align*}
\frac{\left(z^{2}\pi\right)^{j}}{j!}e^{-z^{2}\pi} & =\frac{\left(-1\right)^{j}}{\left(4\pi\right)^{j}j!}\left(2\pi iz\right)^{2j}\mathcal{F}=e^{-z^{2}\pi}\frac{\left(-1\right)^{j}}{\left(4\pi\right)^{j}j!}\mathcal{F}\left(\frac{\mathrm{d}^{2j}}{\mathrm{d}z^{2j}}e^{-z^{2}\pi}\right)\\
 & =\frac{\left(-1\right)^{j}}{\left(4\pi\right)^{j}j!}\mathcal{F}\left(\left(-4\pi\right)^{j}j!e^{-z^{2}\pi}L_{j}^{\left(-\frac{1}{2}\right)}\left(z^{2}\pi\right)\right)=\mathcal{F}\, e^{-z^{2}\pi}L_{j}^{\left(-\frac{1}{2}\right)}\left(z^{2}\pi\right).
\end{align*}
Next, by \eqref{Laguerre} and the identity just derived, 
\begin{align*}
\left(-1\right)^{i}e^{-z^{2}\pi} & L_{i}^{\left(-i-\alpha-\frac{1}{2}\right)}\left(z^{2}\pi\right)=e^{-z^{2}\pi}\sum_{j=0}^{i}\left(-1\right)^{i-j}{-\alpha-\frac{1}{2} \choose i-j}\frac{\left(z^{2}\pi\right)^{j}}{j!}\\
 & =\mathcal{F}\, e^{-z^{2}\pi}\sum_{j=0}^{i}{\alpha-\frac{1}{2}+i-j \choose i-j}L_{j}^{\left(-\frac{1}{2}\right)}\left(z^{2}\pi\right)=\mathcal{F}\, e^{-z^{2}\pi}L_{i}^{\left(\alpha\right)}\left(z^{2}\pi\right),
\end{align*}
the latter identity being a special case ($x=0$, $\beta=-\frac{1}{2}$)
of the more general identity 
\[
\sum_{j=0}^{i}L_{i-j}^{\left(\alpha\right)}\left(x\right)L_{j}^{\left(\beta\right)}\left(y\right)=L_{i}^{\left(\alpha+\beta+1\right)}\left(x+y\right).
\]
 This proves the first statement.

The other statements are an immediate consequence of this first one
and Theorem \ref{Kummer als Laguerre}.
\end{proof}

\subsection{Poisson Summation Formula.}

Given the notation introduced Poisson's summation formula reads $\sum_{z\in\mathbb{Z}}f\left(z\right)=\sum_{k\in\mathbb{Z}}\hat{f}\left(k\right)$
(cf. \citep{Zygmund,Pinsky}), or a bit more generally 
\begin{equation}
\sum_{z\in\mathbb{Z}}f\left(\left(z-z_{0}\right)x\right)e^{-2\pi i\left(z-z_{0}\right)k_{0}}=\frac{1}{x}\sum_{k\in\mathbb{Z}}\hat{f}\left(\frac{k+k_{0}}{x}\right)e^{-2\pi ikz_{0}}\label{eq:Poisson Summation Formula}
\end{equation}
when applied to the same function $f$ but translated ($z_{0}$),
contracted ($x$) and shifted in phase ($k_{0}$). 

As already mentioned, the Poisson summation formula often transforms
slowly converging series into very rapidly converging series. We will
exploit this fact to obtain the following results, which are an application
of Poisson's summation formula to the Fourier transforms elaborated
in the latter section.

\section{Application to Riemann Zeta Function}

It turns out that a particular application of summation techniques
outlined above is a very rapidly converging series for the Riemann
zeta function. We are even able to prove these following variants,
and Riemann's function equation follows as a by-product: 
\begin{thm}
\label{thm:Representations for Zeta}For any $s\in\mathbb{C}$ we
have

\begin{equation}
s\left(s-1\right)\frac{\zeta\left(s\right)\Gamma\left(\frac{s}{2}\right)}{\pi^{\frac{s}{2}}}=1-s\left(1-s\right)\sum_{k=1}\frac{\Gamma\left(\frac{s}{2},k^{2}\pi\right)}{\left(k^{2}\pi\right)^{\frac{s}{2}}}+\frac{\Gamma\left(\frac{1-s}{2},k^{2}\pi\right)}{\left(k^{2}\pi\right)^{\frac{1-s}{2}}}.\label{ZetaWelle}
\end{equation}
More generally, for every arbitrarily chosen $x$ satisfying $\Re\left(x\right)>\left|\Im\left(x\right)\right|$
we find 
\begin{align*}
s\left(s-1\right)\frac{\zeta\left(s\right)\Gamma\left(\frac{s}{2}\right)}{\pi^{\frac{s}{2}}}= & \left(1-s\right)x^{s}+s\, x^{s-1}+\\
 & +s\left(s-1\right)\sum_{k=1}\frac{\Gamma\left(\frac{s}{2},k^{2}x^{2}\pi\right)}{\left(k^{2}\pi\right)^{\frac{s}{2}}}+\frac{\Gamma\left(\frac{1-s}{2},\frac{k^{2}\pi}{x^{2}}\right)}{\left(k^{2}\pi\right)^{\frac{1-s}{2}}};
\end{align*}
moreover
\[
\left(2^{s}-1\right)\left(1-2^{1-s}\right)\frac{\zeta\left(s\right)\Gamma\left(\frac{s}{2}\right)}{\pi^{\frac{s}{2}}}=\Upsilon_{x}\left(s\right)+\Upsilon_{\frac{1}{x}}\left(1-s\right),
\]
where $\Upsilon_{x}\left(s\right)$ is the $2^{\text{nd}}$ forward
difference 
\[
\Upsilon_{x}\left(s\right)=\sum_{k=0}\frac{\Gamma\left(\frac{s}{2},\frac{\left(4k+1\right)^{2}\pi x^{2}}{4}\right)}{\left(\frac{\left(4k+1\right)^{2}\pi}{4}\right)^{\frac{s}{2}}}-2\frac{\Gamma\left(\frac{s}{2},\frac{\left(4k+2\right)^{2}\pi x^{2}}{4}\right)}{\left(\frac{\left(4k+2\right)^{2}\pi}{4}\right)^{\frac{s}{2}}}+\frac{\Gamma\left(\frac{s}{2},\frac{\left(4k+3\right)^{2}\pi x^{2}}{4}\right)}{\left(\frac{\left(4k+3\right)^{2}\pi}{4}\right)^{\frac{s}{2}}}.
\]
\end{thm}
\begin{rem}
It should be stressed that these are globally convergent series for
a function analytic in the \emph{entire }plane, converging in particular
for $s=1$. Moreover, by de l'Hôpital's rule, $\lim\frac{\Gamma\left(s,z\right)}{z^{s-1}e^{-z}}=\lim\frac{z^{s-1}e^{-z}}{z^{s-1}e^{-z}-\left(s-1\right)z^{s-2}e^{-z}}=1$,
the rate of convergence thus is of order 
\begin{eqnarray*}
\mathcal{O}\left(\sum_{k'=k}^{\infty}e^{-k'^{2}\pi}\right) & = & \mathcal{O}\left(e^{-k^{2}\pi}\right)\\
 & = & \mathcal{O}\left(0.04321\dots^{k^{2}}\right),
\end{eqnarray*}
which is pretty quick.

As an additional result, Riemann's function equation follows immediately
from equation \eqref{ZetaWelle}.

It seems that an identity close to \eqref{ZetaWelle} already was
known to Riemann himself, although the proof being based on integrals
involving the Jacobi function $\theta_{3}\left(z\right):=\sum_{k\in\mathbb{Z}}e^{-k^{2}\text{\ensuremath{\pi}z}}$,
or rather the function $\tilde{\theta}\left(z\right):=\theta_{3}'\left(z\right)+\frac{3}{2}\theta_{3}''\left(z\right)$:
Both satisfy $\theta_{3}\left(z\right)=\frac{1}{\sqrt{z}}\theta_{3}\left(\frac{1}{z}\right)$,
but $\theta_{3}^{\left(k\right)}\left(z\right)\rightarrow0$ as $z\rightarrow0$
and $z\rightarrow\infty$ for \emph{any }$k\in\left\{ 0,\,1,\,2,\dots\right\} $.
\footnote{By the way: Suppose that $\theta\left(z\right)=z^{-\alpha-1}\cdot\theta\left(\frac{1}{z}\right),$
then $\left(-1\right)^{i}\frac{z^{1+\alpha+i}}{i!}\theta\left(z\right)=\sum_{j=0}^{i}{i+\alpha \choose i-j}\frac{1}{z^{j}j!}\theta^{\left(j\right)}\left(\frac{1}{z}\right)$,
strikingly reminding to the explicit representation of Laguerre polynomials
\eqref{Laguerre}.%
} By exploiting this fact \eqref{ZetaWelle} can be recovered as well
-- this is for example demonstrated in the third method (out of 7)
of proving the function equation in \citep{Titchmarsh}.
\end{rem}

\begin{rem}
The famous, fast algorithm for multiple evaluations of the Riemann
Zeta function in \citep{Odlyzko+Schoenhage} is based on an a representation
for the zeta function, which is not converging on the entire complex
plane, and the rate of convergence being slower, but allowing precise
bounds and estimates. 

The representation given by \citep{Sondow}, $\zeta\left(s\right)=\frac{1}{1-2^{1-s}}\sum_{n=1}\sum_{k=0}^{n}\frac{\left(-1\right)^{k-1}}{\left(k+1\right)^{s}}{n \choose k}$,
converges globally as well, however, the rate of convergence is significantly
slower. 

A similar algorithm is proposed in \citep{Borwein} with exponential
convergence rate.

So to summarize and compare the representations described in Theorem
\ref{thm:Representations for Zeta} are faster, the only downside
is that $\Gamma\left(\frac{s}{2}\right)$ is very small for huge imaginary
parts, and incomplete Gamma functions have to be evaluated. However,
this can be done very quickly, as will be further outlined below.\end{rem}
\begin{proof}
As for the proof notice first that 
\[
\sum_{z=-\infty}^{\infty}e^{-z^{2}\pi x^{2}}M\left(a,b,z^{2}\pi x^{2}\right)=\frac{1}{x}\frac{\Gamma\left(b\right)}{\Gamma\left(b-a\right)}\sum_{z=-\infty}^{\infty}e^{-\frac{z^{2}\pi}{x^{2}}}U\left(a,a+\frac{3}{2}-b,\frac{z^{2}\pi}{x^{2}}\right),
\]
or, using \eqref{Kummer2},
\[
\begin{array}{l}
1+2\sum_{z=1}^{\infty}e^{-z^{2}\pi x^{2}}M\left(a,b,z^{2}\pi x^{2}\right)\\
\quad=\frac{1}{x}\frac{\Gamma\left(b\right)}{\Gamma\left(b-a\right)}\left(\frac{\Gamma\left(b-\frac{1}{2}-a\right)}{\Gamma\left(b-\frac{1}{2}\right)}+2\sum_{z=1}^{\infty}e^{-\frac{z^{2}\pi}{x^{2}}}U\left(a,a+\frac{3}{2}-b,\frac{z^{2}\pi}{x^{2}}\right)\right).
\end{array}
\]
We choose $a=1$ and $b=1+\frac{s}{2}$. Thus, using \eqref{eq:GammaLower}
and \eqref{eq:GammaUpper}, 
\[
1+2\frac{s}{2}\sum_{z=1}^{\infty}\frac{\gamma\left(\frac{s}{2},z^{2}\pi x^{2}\right)}{\left(z^{2}\pi x^{2}\right)^{\frac{s}{2}}}=\frac{s}{\left(s-1\right)x}+\frac{2}{x}\frac{\Gamma\left(1+\frac{s}{2}\right)}{\Gamma\left(\frac{s}{2}\right)}\sum_{z=1}^{\infty}\frac{\Gamma\left(\frac{1-s}{2},\frac{z^{2}\pi}{x^{2}}\right)}{\left(\frac{z^{2}\pi}{x^{2}}\right)^{\frac{1-s}{2}}},
\]
or 
\begin{equation}
\frac{x^{s}}{s}+\frac{x^{s-1}}{1-s}+\sum_{z=1}^{\infty}\frac{\gamma\left(\frac{s}{2},z^{2}\pi x^{2}\right)}{\left(z^{2}\pi\right)^{\frac{s}{2}}}=\sum_{z=1}^{\infty}\frac{\Gamma\left(\frac{1-s}{2},\frac{z^{2}\pi}{x^{2}}\right)}{\left(z^{2}\pi\right)^{\frac{1-s}{2}}}.\label{eq:Fourier for Gamma}
\end{equation}
 This nice identity at hand we find that 
\begin{align*}
s\left(s-1\right)\frac{\zeta\left(s\right)\Gamma\left(\frac{s}{2}\right)}{\pi^{\frac{s}{2}}} & =s\left(s-1\right)\sum_{k=1}\frac{\gamma\left(\frac{s}{2};k^{2}\pi x^{2}\right)+\Gamma\left(\frac{s}{2};k^{2}\pi x^{2}\right)}{\left(k^{2}\pi\right)^{\frac{s}{2}}}\\
 & =\left(1-s\right)x^{s}+s\, x^{s-1}\\
 & -s\left(1-s\right)\sum_{k=1}\frac{\Gamma\left(\frac{1-s}{2};\frac{k^{2}\pi}{x^{2}}\right)}{\left(k^{2}\pi\right)^{\frac{1-s}{2}}}+\frac{\Gamma\left(\frac{s}{2};k^{2}\pi x^{2}\right)}{\left(k^{2}\pi\right)^{\frac{s}{2}}},
\end{align*}
which is the second statement. Notice, that the condition $\Re\left(x\right)>\left|\Im\left(x\right)\right|$
insures both, $\Re\left(x^{2}\right)>0$ and $\Re\left(\frac{1}{x^{2}}\right)>0$,
which is necessary for convergence.

The first statement is obvious by the choice $x=1$.

As for the next statement recall that 
\[
\left(1-2^{1-s}\right)\frac{\zeta\left(s\right)\Gamma\left(\frac{s}{2}\right)}{\pi^{\frac{s}{2}}}=\sum_{k=1}\left(-1\right)^{k-1}\frac{\left(\gamma\left(\frac{s}{2},k^{2}\pi x^{2}\right)+\Gamma\left(\frac{s}{2},k^{2}\pi x^{2}\right)\right)}{\left(k^{2}\pi\right)^{\frac{s}{2}}}.
\]
In order to get rid of slowly converging $\gamma$ (and replace it
by rapidly converging Fourier transform $\Gamma$) we apply Poisson
summation formula \eqref{eq:Poisson Summation Formula} again, now
with $z_{0}=\frac{1}{2}$ and $k_{0}=0$. Hence, 
\begin{align*}
\left(1-2^{1-s}\right)\frac{\zeta\left(s\right)\Gamma\left(\frac{s}{2}\right)}{\pi^{\frac{s}{2}}} & =\frac{x^{s}}{s}-\sum_{k=1}\frac{\Gamma\left(\frac{1-s}{2},\frac{\left(k-\frac{1}{2}\right)^{2}\pi}{x^{2}}\right)}{\left(\left(k-\frac{1}{2}\right)^{2}\pi\right)^{\frac{1-s}{2}}}+\left(-1\right)^{k}\frac{\Gamma\left(\frac{s}{2},k^{2}\pi x^{2}\right)}{\left(k^{2}\pi\right)^{\frac{s}{2}}}\\
 & =\frac{x^{s}}{s2^{s}}-\frac{1}{2^{s-1}}\sum_{k=1}\frac{\Gamma\left(\frac{1-s}{2},\frac{\left(2k-1\right)^{2}\pi}{x^{2}}\right)}{\left(\left(2k-1\right)^{2}\pi\right)^{\frac{1-s}{2}}}+\left(-1\right)^{k}\frac{\Gamma\left(\frac{s}{2},\frac{k^{2}\pi x^{2}}{4}\right)}{\left(k^{2}\pi\right)^{\frac{s}{2}}},
\end{align*}
as the statement holds for $\frac{x}{2}$ as well. Combining the latter
two identities and rearranging the terms is cumbersome, but finally
gives the assertion.\end{proof}
\begin{rem}
$x:=\frac{s}{s-1}$ is notably a possible choice if $\Im\left(s\right)>\frac{1+\sqrt{2}}{2}\approx1.21$,
as in this case $\Re\left(x\right)>\left|\Im\left(x\right)\right|$.
The advantage of this this particular choice is that the term $\left(1-s\right)x^{s}+s\, x^{s-1}$
vanishes.
\end{rem}

\begin{rem}
It should be noticed that the identity \eqref{eq:Fourier for Gamma}
is central here. It allows to replace the slowly converging series
(which involves $\gamma$) by its Fourier transform, which is the
very fast converging series (which involves $\Gamma$). This is the
key strategy in all improvements of convergence above.
\end{rem}

\section{Asymptotics of the Laguerre Polynomials}
\begin{thm}
\label{Laguerre Asymptotics}(Asymptotics of the Laguerre Polynomial)
Let $\Re\left(z\right)>0$. Then, as $i\rightarrow\infty$,

$L_{i}^{\left(\alpha\right)}\left(z\right)\approx\frac{i^{\frac{\alpha}{2}-\frac{1}{4}}}{\sqrt{\pi}}\frac{e^{\frac{z}{2}}}{z^{\frac{\alpha}{2}+\frac{1}{4}}}\cos\left(2\sqrt{z\left(i+\frac{\alpha+1}{2}\right)}-\frac{\pi}{2}\left(\alpha+\frac{1}{2}\right)\right)$
and 

$L_{i}^{\left(\alpha\right)}\left(-z\right)\approx\frac{i^{\frac{\alpha}{2}-\frac{1}{4}}}{\sqrt{\pi}}\frac{e^{-\frac{z}{2}}}{z^{\frac{\alpha}{2}+\frac{1}{4}}}\exp\left(2\sqrt{z\left(i+\frac{\alpha+1}{2}\right)}\right)$.\end{thm}
\begin{rem}
We write $f_{i}\approx g_{i}$ as $i\to\infty$ to indicate that $\lim_{i\to\infty}\frac{f_{i}}{g_{i}}=1$.
\end{rem}

\begin{rem}
To prove the statements of the theorem we will involve Bessel functions.
Bessel functions, of first and second kind, are 
\[
J_{\alpha}\left(x\right):=\sum\frac{\left(-1\right)^{m}}{m!\Gamma\left(\alpha+m+1\right)}\left(\frac{x}{2}\right)^{\alpha+2m}\text{ and }I_{\alpha}\left(x\right):=\sum\frac{1}{m!\Gamma\left(\alpha+m+1\right)}\left(\frac{x}{2}\right)^{\alpha+2m}.
\]
Notice, that $\frac{J_{\alpha}\left(ix\right)}{\left(ix\right)^{\alpha}}=\frac{I_{\alpha}\left(x\right)}{x^{\alpha}}$
are entire, even functions. Moreover, Kummer's second formula links
Bessel functions to confluent hypergeometric functions, as $\frac{J_{\alpha}\left(z\right)}{\left(\frac{z}{2}\right)^{\alpha}}=e^{-z}\frac{M\left(a+\frac{1}{2},2a+1,2z\right)}{\Gamma\left(a+1\right)}.$\end{rem}
\begin{proof}
We start with identity 13.3.7 from \citep{abramowitz+stegun} which
states that 
\begin{equation}
\frac{M\left(a,b,z\right)}{\Gamma\left(b\right)}=e^{\frac{z}{2}}\cdot\sum_{n=0}A_{n}\cdot\left(\frac{z}{2}\right)^{n}\frac{J_{b-1+n}\left(\sqrt{2z\left(b-2a\right)}\right)}{\left(\frac{1}{2}\sqrt{2z\left(b-2a\right)}\right)^{b-1+n}},\label{eq:M als Bessel}
\end{equation}
where $J$ is the Bessel function of the first kind and $A$ satisfies
the recursion $A_{0}:=1$, $A_{1}:=0$, $A_{2}:=\frac{b}{2}$ and
$A_{n}:=\left(1+\frac{b-2}{n}\right)A_{n-2}-\frac{b-2a}{n}A_{n-3}$.
Although we refer to this Identity \eqref{eq:M als Bessel} without
proof we mention that it follows straight forward by successively
comparing the respective coefficients of $z$ in \eqref{eq:M als Bessel}.
The coefficients satisfy $A_{n}=\mathcal{O}\left({a-b \choose n}\right)$
for $b>2a$, as follows directly from the recursive definition. 

Following \citep{Arfken+Weber}, 
\begin{eqnarray}
J_{\alpha}\left(x\right) & \approx & \sqrt{\frac{2}{\pi x}}\cos\left(x-\frac{\pi}{2}\left(\alpha+\frac{1}{2}\right)\right),\quad I_{\alpha}\left(x\right)\approx\frac{e^{x}}{\sqrt{2\pi x}}\label{eq:AsymptoticBessel}
\end{eqnarray}
as $x\rightarrow\infty$. 

Combining all those ingredients we see that the initial term ($n=0$)
dominates all other summands in \eqref{eq:M als Bessel} when $a\rightarrow-\infty$,
so we may neglect them for $n\ge1$ to identify the rate of convergence.
But for $n=0$ the assertion states that 
\[
\frac{M\left(a,b,z\right)}{\Gamma\left(b\right)}\approx e^{\frac{z}{2}}\cdot\frac{J_{b-1}\left(\sqrt{2z\left(b-2a\right)}\right)}{\left(\frac{1}{2}\sqrt{2z\left(b-2a\right)}\right)^{b-1}},
\]
which is a useful approximation itself.

There is another, similar approach, which is quite useful, which we
want to give here as well for the sake of completeness and further
reference: It starts with identity 13.3.8 from \citep{abramowitz+stegun}
with parameter $h=\frac{1}{2}$. This reads 
\[
\frac{M\left(a,b,z\right)}{\Gamma\left(b\right)}=e^{\frac{z}{2}}\cdot\sum_{n=0}C_{n}\, z^{n}\frac{J_{b-1+n}\left(2\sqrt{-az}\right)}{\left(\sqrt{-az}\right)^{b-1+n}},
\]
 where $C_{0}=1$, $C_{1}=-b$, $C_{2}=\frac{b\left(b+1\right)}{2}$
and $C_{n}=-\frac{b}{2n}C_{n-1}+\left(\frac{1}{4}+\frac{b-2}{4n}\right)C_{n-2}+\frac{a}{4n}C_{n-3}$.
Here, $C_{n}=\mathcal{O}\left(\frac{1}{2^{n}}{a-b \choose n}\right)$
and similar to the reasoning above we obtain 
\[
\frac{M\left(a,b,z\right)}{\Gamma\left(b\right)}\approx e^{\frac{z}{2}}\cdot\frac{J_{b-1}\left(2\sqrt{-az}\right)}{\left(\sqrt{-az}\right)^{b-1}}.
\]

From both asymptotic identities the theorem follows in view of \eqref{Laguerre},
that is 
\[
L_{i}^{(\alpha)}\left(z\right):={i+\alpha \choose i}\, M\left(-i,\alpha+1,z\right),
\]
and \eqref{eq:AsymptoticBessel} as $i\to\infty$.\end{proof}
\begin{rem}
We want to stress that the method used in the proof above allows to
compute terms of higher order of the expressions given in Theorem
\ref{Laguerre Asymptotics}. The higher order correction terms give
successive improvements of order $\frac{1}{\sqrt{i}}$. 

For an interesting treatment to evaluate Laguerre polynomials for
large $i$ we refer the reader to \citep{Borwein+Borwein+Crandall}. 
\end{rem}

\section{Applications to the incomplete Gamma Function}

In order to make use of the rapidly converging series \eqref{ZetaWelle}
it is necessary to have a good implementation for the upper incomplete
gamma function at hand. To this end we further elaborate on a procedure
based on continued fractions (cf. \citep{Akiyama+Tanigawa}), which
is always a good candidate for rapid convergence. The formula is a
special case of Gauss' continued fraction method using confluent hypergeometric
functions (see \citep{Jones+Thron,Wall}): 
\begin{equation}
\Gamma\left(s,z\right)=\cfrac{z^{s}e^{-z}}{z+\cfrac{1-s}{1+\cfrac{1}{z+\cfrac{2-s}{1+\cfrac{2}{z+\cfrac{3-s}{1+\cfrac{3}{z+\ddots}}}}}}}\label{eq:IncompleteGammacontinuousFraction}
\end{equation}
(as a \emph{formal }power series this is equivalent to $\Gamma\left(s,z\right)=z^{s}e^{-z}\cdot{}_{2}F_{0}\left(1,1-s,-\frac{1}{z}\right)).$

Stopping a general continued fraction gives its convergents, the \emph{$k^{\text{th}}$-convergent}
is 
\[
\frac{p_{k}}{q_{k}}:=\cfrac{a_{1}}{b_{1}+\cfrac{a_{2}}{b_{2}+\cfrac{\ddots}{\ddots+\cfrac{a_{k-1}}{b_{k-1}+\cfrac{a_{k}}{b_{k}}}}}}.
\]

To simplify the respective values there is -- from standard theory
\citep{Wall} -- the recursion 
\[
\left(\begin{array}{c}
p_{k}\\
q_{k}
\end{array}\right)=a_{k}\left(\begin{array}{c}
p_{k-2}\\
q_{k-2}
\end{array}\right)+b_{k}\left(\begin{array}{c}
p_{k-1}\\
q_{k-1}
\end{array}\right)
\]
with initial conditions $\left(\begin{array}{c}
p_{0}\\
q_{0}
\end{array}\right):=\left(\begin{array}{c}
0\\
1
\end{array}\right)$ and $\left(\begin{array}{c}
p_{1}\\
q_{1}
\end{array}\right)=\left(\begin{array}{c}
a_{1}\\
b_{1}
\end{array}\right)$, and additionally the somewhat more explicit formula
\begin{equation}
\frac{p_{k}}{q_{k}}=\sum_{i=1}^{k}\left(-1\right)^{i-1}\frac{a_{1}\cdot a_{2}\cdot\dots a_{i}}{q_{i}\cdot q_{i-1}}.\label{eq:continued fraction 3}
\end{equation}
When applied to the continuous fraction of the incomplete gamma function
\eqref{eq:IncompleteGammacontinuousFraction} the recursions for the
numerator and denominator $\frac{p_{k}\left(s,z\right)}{q_{k}\left(s,z\right)}$
read 
\[
p_{k}\left(s,z\right)=\begin{cases}
\left(\frac{k}{2}-s\right)p_{k-2}\left(s,z\right)+p_{k-1}\left(s,z\right) & k\mbox{ even}\\
\frac{k-1}{2}p_{k-2}\left(s,z\right)+z\, p_{k-1}\left(s,z\right) & k\mbox{ odd}
\end{cases}
\]
 and 
\begin{equation}
q_{k}\left(s,z\right)=\begin{cases}
\left(\frac{k}{2}-s\right)q_{k-2}\left(s,z\right)+q_{k-1}\left(s,z\right) & k\mbox{ even}\\
\frac{k-1}{2}q_{k-2}\left(s,z\right)+z\, q_{k-1}\left(s,z\right) & k\mbox{ odd},
\end{cases}\label{eq:Recursion q}
\end{equation}
but different initial values: $p_{0}=0$, $p_{1}=1$, $q_{-1}=0$
and $q_{0}=1$. It comes without surprise that Laguerre polynomials
appear again, as these recursions have the closed expression 
\[
q_{2k}\left(s,z\right)=k!L_{k}^{\left(-s\right)}\left(-z\right)\mbox{ and }q_{2k+1}\left(s,z\right)=k!\, z\, L_{k}^{\left(1-s\right)}\left(-z\right).
\]

In view of \eqref{eq:continued fraction 3} the knowledge of the denominator
is already sufficient to get the convergents, in explicit terms we
get the simple expressions 
\[
\frac{p_{2k}\left(s,z\right)}{q_{2k}\left(s,z\right)}=\sum_{i=0}^{k-1}\frac{\frac{1}{i+1}{i-s \choose i}}{L_{i}^{\left(-s\right)}\left(-z\right)L_{i+1}^{\left(-s\right)}\left(-z\right)}\rightarrow\frac{\Gamma\left(s,z\right)}{z^{s}e^{-z}}
\]
 and 
\begin{equation}
\frac{p_{2k+1}\left(s,z\right)}{q_{2k+1}\left(s,z\right)}=\frac{1}{z}+\frac{1}{z}\sum_{i=1}^{k}\frac{{i-s \choose i}}{L_{i-1}^{\left(1-s\right)}\left(-z\right)L_{i}^{\left(1-s\right)}\left(-z\right)}\rightarrow\frac{\Gamma\left(s,z\right)}{z^{s}e^{-z}}\label{eq:continued fraction 2}
\end{equation}
which have been observed in \citep{Borwein+Borwein+Crandall}. To
find a handy expression for the numerator as well is surprisingly
much more difficult -- but yes, the somewhat curious result involves
Laguerres again:
\begin{thm}
For any $s$ and $z\ne0$ we have
\begin{eqnarray}
\frac{\Gamma\left(s,z\right)}{z^{s}e^{-z}} & = & \lim_{k\to\infty}\frac{\sum_{i=1}^{\left\lfloor \frac{k+1}{2}\right\rfloor }L_{k-2i+1}^{\left(2i-s\right)}\left(-z\right)\frac{{s-1 \choose i-1}}{{k-1 \choose i-1}}}{kL_{k}^{\left(-s\right)}\left(-z\right)}\nonumber \\
 & = & \frac{1}{z}+\lim_{k\to\infty}\frac{\sum_{i=1}^{\left\lfloor \frac{k+1}{2}\right\rfloor }L_{k-2i+1}^{\left(2i+1-s\right)}\left(-z\right)\frac{{s-1 \choose i}}{{k \choose i}}}{z\, L_{k}^{\left(1-s\right)}\left(-z\right)},\label{eq:Gamma als Laguerre}
\end{eqnarray}
the rate of approximation for $\Re\left(z\right)>0$ being of order
\textup{$\mathcal{O}\left(\frac{e^{-4\sqrt{kz}}}{\sqrt{kz}}\right)$}.\end{thm}
\begin{rem*}
The integer valued floor function $\left\lfloor .\right\rfloor $
satisfies $x-1<\left\lfloor x\right\rfloor \le x$.\end{rem*}
\begin{proof}
The fractions in the limit represent explicit expressions for $\frac{p_{2k}\left(s,z\right)}{q_{2k}\left(s,z\right)}$
($\frac{p_{2k+1}\left(s,z\right)}{q_{2k+1}\left(s,z\right)}$, resp.).
Again, the proof links the recursions \eqref{eq:Recursion q} introduced
above to well known recursions of Laguerre polynomials.

The rate of approximation is an interesting consequence of \eqref{eq:continued fraction 3}
(same for \eqref{eq:continued fraction 2}):

\begin{eqnarray*}
\mathcal{O}\left(\sum_{i=k}^{\infty}\frac{\frac{1}{i+1}{i-s \choose i}}{L_{i}^{\left(-s\right)}\left(-z\right)L_{i}^{\left(1-s\right)}\left(-z\right)}\right) & = & \mathcal{O}\left(\sum_{i=k}^{\infty}\frac{i^{-s-1}}{i^{-\frac{s}{2}-\frac{1}{4}}i^{\frac{1-s}{2}-\frac{1}{4}}e^{4\sqrt{iz}}}\right)\\
 & = & \mathcal{O}\left(\int_{k}^{\infty}\frac{i^{-1}}{e^{4\sqrt{iz}}}\mathrm{d}i\right)\\
 & = & \mathcal{O}\left(\Gamma\left(0,4\sqrt{kz}\right)\right)\\
 & = & \mathcal{O}\left(\frac{e^{-4\sqrt{kz}}}{\sqrt{kz}}\right),
\end{eqnarray*}
which is finally the desired rate.
\end{proof}

\section{Application to Riemann's Zeta Function}

We have given a few approximations for the upper incomplete Gamma
function which involve polynomials, for example \eqref{eq:Gamma als Laguerre-Reihe},
\eqref{eq:continued fraction 3} (\eqref{eq:continued fraction 2},
respectively) and \eqref{eq:Gamma als Laguerre}. Those expressions
converge sufficiently quick and do not impose any difficulties as
their argument $z$ tends to infinity. So they can be substituted
in \eqref{ZetaWelle} to give recent approximations and a variety
of possible investigations on the geometry of the Riemann zeta function. 

As an initial example combine \eqref{eq:Gamma als Laguerre-Reihe}
and \eqref{ZetaWelle} to get the representation
\[
\tilde{\zeta}_{n}^{\left(\Delta\right)}\left(s\right):=1-s\left(1-s\right)\sum_{k=0}^{n-1}a_{k}^{\left(\Delta\right)}\left(\frac{1}{{\frac{s}{2}+k+\Delta \choose k+1}}+\frac{1}{{\frac{1-s}{2}+k+\Delta \choose k+1}}\right)\xrightarrow[n\to\infty]{}s\left(s-1\right)\frac{\zeta\left(s\right)\Gamma\left(\frac{s}{2}\right)}{\pi^{\frac{s}{2}}},
\]
where $a_{k}^{\left(\Delta\right)}:=\frac{\sum_{z=1}e^{-z^{2}\pi}L_{k}^{\left(\Delta-\frac{1}{2}\right)}\left(z^{2}\pi\right)}{k+1}$.
Notice, that this representation converges, as $n\to\infty$, uniformly
on $-2\Delta-1<\Re\left(s\right)\le2\,\Delta$. 

Moreover -- and this is a key observation -- the numerator or the
function 
\[
\tilde{\zeta}_{n}^{\left(\Delta\right)}\left(s\right){\frac{s}{2}+n-1+\Delta \choose n}{\frac{1-s}{2}+n-1+\Delta \choose n},
\]
as a function in variable $s$, is a \emph{polynomial }of degree $2n$.
Thus, by Hurwitz' Theorem in complex analysis (cf. \citep{Conway}),
the zeros of $\zeta$ are just the accumulation points of $\tilde{\zeta}_{n}^{\left(\Delta\right)}$'s
zeros. 

Interestingly, the zeros of the polynomials above have a very nice,
symmetric pattern in common, a typical result is plotted in figure
\ref{Flo:Zeta25} %
\footnote{Figures have been computed using Mathematica.%
}: 
\begin{figure}
\includegraphics{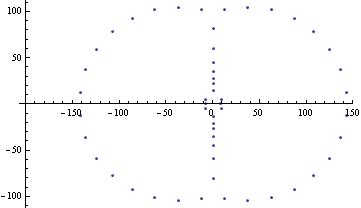}

\caption{Zeros of $\tilde{\zeta}_{25}^{\left(5\right)}$.}
\label{Flo:Zeta25}
\end{figure}
The total of zeros is 50, 18 on the left (right) of $\Re\left(s\right)=\frac{1}{2}$,
7 have a positive (negative) imaginary part on the critical line,
a few exemptions occur close to $2\Delta$ ($1-2\Delta$, respectively). 

However: all zeros in the area of convergence, which is $-2\Delta-1<\Re\left(s\right)\le2\,\Delta$,
lie on the critical line. 

Another pattern is found when employing \eqref{eq:continued fraction 3}
or \eqref{eq:Gamma als Laguerre}, in our next example for $k=6$
and involving $z^{2}\pi$ in $a_{k}^{\left(\Delta\right)}$ for $z$
from 1 to 5: The corresponding polynomial has degree 48, its zeros
are depicted in figure \ref{Flo:Zeta25-1} (6 do not fit to the scale
chosen).

\begin{figure}
\includegraphics{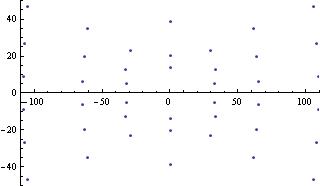}

\caption{Zeros of $\tilde{\zeta}_{48}^{\left(5,6\right)}$.}
\label{Flo:Zeta25-1}
\end{figure}

Interestingly, the most convenient pattern is observed when substituting
\eqref{eq:Gamma als Laguerre-Reihe} into the equation $\Upsilon_{x}\left(s\right)$.
Stopping the infinite sum including $\frac{L_{k}^{\Delta}}{\left(k+1\right){k+1+a-\frac{s}{2} \choose i+1}}$
as a final term and solving $\Upsilon_{1}\left(s\right)+\Upsilon_{1}\left(1-s\right)=0$
again leads to finding the roots of polynomials in $s$, which a of
degree $2k$ here. Obviously, the zeros corresponding to the factor
$\left(2^{s}-1\right)\left(1-2^{1-s}\right)$ have to become visible
as well, but, as numerical experiments show, those zeros are being
added very slowly, as $k$ increases. We have depicted some zeros
for $k=50$ in figure \ref{Flo:Zeros of ZetaYpsilon}, where only
4 zeros corresponding to $s=\frac{1}{2}\pm\frac{1}{2}\pm\frac{2\pi i}{\ln2}$
($\frac{2\pi}{\ln2}\approx9.06$) appear. However, and this is potentially
a big advantage in comparison to the other figures, all other figures
are located on the critical line and two symmetric bubbles except
a few others usually located close to the real line.
\begin{figure}
\includegraphics{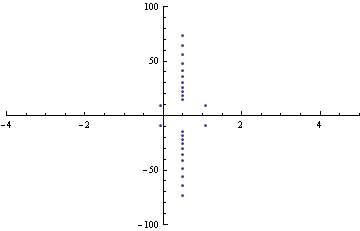}

\caption{Zeros of $\tilde{\zeta}_{50}$, 4 zeros corresponding to the factor
$\left(2^{s}-1\right)\left(1-2^{1-s}\right)$.}
\label{Flo:Zeros of ZetaYpsilon}
\end{figure}
The complete pattern of zeros is depicted in figure \ref{Flo:ZetaYpsilon2}.
\begin{figure}
\includegraphics{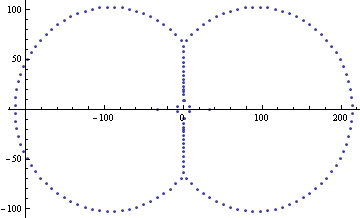}

\caption{Zeros of $\tilde{\zeta}_{50}$, $\alpha=5$}
\label{Flo:ZetaYpsilon2}
\end{figure}
Do these associated polynomials always have their zeros in the region
of convergence on the critical strip? The answer ``Yes'' obviously
implies RH. So this question seems being worth an attempt and has
to be investigated in much more detail in promising, further research.

\section{Acknowledgment}

We would like to thank the reviewers for their significant contribution
to improve the paper.

The plots have been produced by use of Mathematica.

\begin{figure}[h]
\centering{}\includegraphics[scale=0.3]{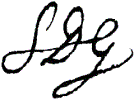}
\end{figure}

\bibliographystyle{alpha}
\phantomsection\addcontentsline{toc}{section}{\refname}\bibliography{LiteraturAlois}

\end{document}